\documentclass[pdflatex,sn-mathphys-num]{sn-jnl}

\usepackage{graphicx}%
\usepackage{multirow}%
\usepackage{amsmath,amssymb,amsfonts}%
\usepackage{amsthm}%
\usepackage{mathrsfs}%
\usepackage[title]{appendix}%
\usepackage{xcolor}%
\usepackage{textcomp}%
\usepackage{manyfoot}%
\usepackage{booktabs}%
\usepackage{algorithm}%
\usepackage{algorithmicx}%
\usepackage{algpseudocode}%
\usepackage{listings}%
\usepackage[bb=boondox]{mathalpha}
\usepackage{mathtools, amssymb}
\usepackage{xparse}
\usepackage{xcolor}%
\usepackage{tikz}%

\newcommand{\real}{\mathbb{R}}

\newcommand{\cB}{\mathcal{B}}

\newcommand{\cL}{\mathcal{L}}

\newcommand{\NN}{\mathbb{N}}

\DeclareFontFamily{U}{ntxmia}{}
\DeclareFontShape{U}{ntxmia}{m}{it}{<-> ntxmia }{}
\DeclareFontShape{U}{ntxmia}{b}{it}{<-> ntxbmia }{}
\DeclareSymbolFont{lettersA}{U}{ntxmia}{m}{it}
\SetSymbolFont{lettersA}{bold}{U}{ntxmia}{b}{it}

\makeatletter
\newcommand*{\rom}[1]{\expandafter\@slowromancap\roMannumeral #1@}
\makeatother

\ExplSyntaxOn
\NewDocumentCommand{\varmathbb}{m}
 {
  \tl_map_inline:nn { #1 }
  {
    \use:c { varbb##1 }
  }
 }
\tl_map_inline:nn { ABCDEFGHIJKLMNOPQRSTUVWXYZ }
 {
  \exp_args:Nc \DeclareMathSymbol{varbb#1}{\mathord}{lettersA}{\int_eval:n { `#1+67 }}
 }
\exp_args:Nc \DeclareMathSymbol{varbbk}{\mathord}{lettersA}{169}
\ExplSyntaxOff

\newcommand{\opT}{{\varmathbb{T}}}

\theoremstyle{thmstyleone}%
\newtheorem{theorem}{Theorem}
\newtheorem{proposition}[theorem]{Proposition}%

\theoremstyle{thmstyletwo}%
\newtheorem{fact}{Fact}
\newtheorem{assum}{Assumption}
\theoremstyle{thmstylethree}%
\newtheorem{corollary}{Corollary}
\newtheorem{lemma}{Lemma}%

\definecolor{auditgreen}{RGB}{0,120,60}
\definecolor{auditred}{RGB}{190,0,0}

\begin{document}

\title[Article Title]{M-duality in fixed-point problems for general normed spaces}


\author[1]{\fnm{Jongmin} \sur{Lee}}\email{jongmlee@engineering.upenn.edu}\equalcont{These authors contributed equally to this work.}

\author[2]{\fnm{Sangyeop} \sur{Woo}}\email{wsy030818@snu.ac.kr}
\equalcont{These authors contributed equally to this work.}

\author[3]{\fnm{Ernest K.} \sur{Ryu}}\email{eryu@math.ucla.edu}

\affil[1]{\orgdiv{Department of Electrical and Systems Engineering}, \orgname{University of Pennsylvania}}

\affil[2]{\orgdiv{Department of Mathematical Sciences}, \orgname{Seoul National University}}

\affil[3]{\orgdiv{Department of Mathematics}, \orgname{UCLA}}



\abstract{Solving fixed-point problems by reducing residual bounds is one of the central approaches in fixed-point iteration, and, in particular, the Halpern iteration is known to effectively reduce the residual. Recently, a new type of duality theorem has been discovered in the convex optimization and fixed-point problem literature and establish a certain symmetry in the convergence behavior of algorithms that are dual to each other. However, these duality theorems rely on a Lyapunov-style proof structure and mostly focus on Hilbert spaces. In this work, we first establish an M-duality theorem for Mann iterations with Lipschitz operators in general normed spaces, assuming only mild conditions on the coefficients of the iterations. Moreover, leveraging this M-duality, we derive new algorithms that attain a strictly smaller worst-case residual bound than that of the Halpern iteration, which was previously known as the state-of-the-art method in general normed spaces.


}


\maketitle
\section{Introduction}\label{sec1}
Given an $\ell$-Lipschitz operator $\varmathbb{T}: \cB \rightarrow \cB$, where $\ell \in (0,1]$, $\cB$ is a nonempty convex domain in a normed space $(X, \|\cdot\|)$, and nonnegative scalars
$\{m^n_i\}_{0 \le i \le n \le N}$, the \emph{Mann iteration} \cite{Mann1953mean} is
\[
x^{n} = \sum^{n}_{i=0} m^n_i \opT x^{i-1},
\qquad \forall\, n=0,\dots,N,
\]
where $\opT x^{-1}=x^0 \in \cB$,  $\sum^n_{i=0}m^n_i=1$, $m^n_n>0$, $\forall 0\le n \le N$, and $N$ is the total iteration count. The Mann iteration is a general class of fixed-point iterations that aims to find a fixed point of the operator by reducing the residual $\|\opT x-x\|$, one of the commonly used performance measures \cite{leustean2007rates, matsushita2017convergence, cominetti2014rate, lee2024accelerating, lee2025multi,d2025, cc2023, bravo2026minimax, bravo2026krasnosel, krasnosel1955two, baillon1992optimal, bc2018}. In particular, we consider the following assumptions on the coefficients of the Mann iteration:
\begin{assum}\label{assum:montone}
For all $0\le i<j\le N$, $m^{j}_i \le m^{j-1}_i$.
\end{assum}
\begin{assum}\label{assum:sim}
For all $0 \le i < j\le N$, $\sum^{j-1}_{l=i}m^j_l \le m^i_i$.
\end{assum}
\noindent These assumptions have been studied in the literature on fixed-point iterations in general normed spaces \cite{bc2018, bravo2021universal, cc2023, bravo2026minimax, bravo2026krasnosel} and are considered mild in the sense that the coefficients of widely studied fixed-point iterations, such as the Halpern iteration \cite{lieder2021convergence, ss2017, pr2022, bravo2026minimax} and KM iteration \cite{krasnosel1955two, baillon1992optimal, bc2018}, are often chosen to satisfy these assumptions.

The optimal transport metrics framework \cite{baillon1992optimal} was developed as a tool for analyzing residual bounds of fixed-point iterations in general normed spaces. Roughly speaking, by interpreting differences between iterates generated by fixed-point iterations as solutions to recursive optimal transport problems and constructing worst-case operators, this framework provides tight bounds on the residual when the Lipschitz factor $\ell$ and coefficients $\{m^n_i\}_{1 \le i \le n \le N}$ are given \cite{bravo2021universal, cc2023, bravo2026minimax}. Therefore, utilizing this optimal transport metrics framework, we can consider the \emph{worst-case residual bound} of the given Mann iteration as follows:
\[
R(\{m^n_i\}_{1 \le i \le n \le N}, \ell)
\triangleq
\sup_{\substack{
\opT\in\cL_\ell,\; x^0\in\cB\\
x^n=\sum^n_{i=0}m^n_i\,\opT x^{i-1},\;\forall\,0\le n\le N
}}
\frac{\|x^N-\opT x^N\|}{\kappa},
\]
where $\cL_\ell$ denotes the class of Lipschitz operators on arbitrary normed spaces and $\kappa=\max_{-1\le i,j\le N}\|\opT x^i-\opT x^j\|$.
This optimal transport bound is universal in the sense that it depends only on the coefficients $\{m^n_i\}_{0\le i \le n \le N}$ and the Lipschitz coefficient $\ell$, and not on the particular map $\opT$ or the initial point $x^{0}$. Moreover, the optimal transport plan obtained by solving the corresponding optimal transport problem could provide a tight analytic bound for a given Mann iteration \cite{bc2018, cc2023, bravo2026minimax, bravo2026krasnosel}.

Recently, several new duality theorems have been discovered that reveal a certain symmetry in the convergence behavior of dual algorithms: H-duality theorems \cite{kim2023time, yoon2024optimal} for gradient descent methods in Euclidean space and fixed-point iterations in a Hilbert space, and the mirror duality theorem \cite{kim2023mirror} for mirror descent methods in non-Euclidean spaces. Specifically, assuming a Lyapunov-style proof structure, \cite{kim2023time, kim2023mirror} establish one-to-one correspondences between first-order algorithms that reduce function values and gradient norms, and \cite{yoon2024optimal} shows that two fixed-point iterations that are H-dual to each other exhibit a common residual bound. Although these duality theorems open a new approach to analyzing the convergence of algorithms, prior duality theorems are somewhat limited as they require a specific Lyapunov-style proof structure to establish the theorem, and most works focus on Hilbert spaces.

\paragraph{Contribution} In this work, we first establish a duality theorem for the Mann iteration with an $\ell$-Lipschitz operator $\opT$ with $\ell \in (0,1]$ in general normed spaces, which we call \emph{Mann duality} (M-duality). We consider the following \emph{M-dual Mann iteration}
\begin{align*}
x^{n+1}
={}& m^{\tiny{N-n}}_{N-n} \opT x^n+ \sum^{n-1}_{j=0}
\left(
\sum^{n-j}_{i=0} m^{N-j}_{N-n+i}
-
\sum^{n-j-1}_{i=0}m^{N-j-1}_{N-n+i}
\right)\opT x^{j}
+
\left(1-\sum^n_{i=0}m^{N}_{N-i}\right)x^0
\end{align*}
for $0 \le n \le N-1$. For simplicity, we denote the dual coefficient corresponding to $m_i^n$ by $(m_i^n)^M$ for all $0 \le i\le n \le N$. Then, we establish the following symmetry of the worst-case residual bounds between two Mann iterations in general normed spaces that are M-dual to each other.
\begin{theorem}\label{thm:m-dual}
    Under Assumptions \ref{assum:montone} and
    \ref{assum:sim},
    
    \[R(\{m^n_i\}_{0 \le i \le n \le N}, \ell)
    =
    R(\{m^n_i\}^M_{0 \le i \le n \le N}, \ell).\]
\end{theorem}
\noindent
As our M-duality theorem clearly states, compared to prior duality theorems, our theorem does not require any Lyapunov proof structure and holds in general normed spaces. We believe that this simplicity comes from the fact that the optimal transport metric framework relies mainly on the triangle inequality and properties of norms, whereas prior analyses involve more complicated calculations through inner products in Hilbert spaces. 

To briefly describe the proof sketch, we first define the \emph{consecutive quadrangle term} $\Delta^2 d_{i,j}$ within the optimal transport metric framework and explicitly characterize $\Delta^2 d_{i,j}$ and an optimal transport plan by induction, as in Proposition~\ref{lem:sim}. Then, focusing on the first-order term in Proposition~\ref{lem:sim} and its coefficients $c_{i,j}$, we establish a new formula for the residual bound and a recursive relation for $c_{i,j}$. Notably, we identify an M-dual symmetry between $c_{i,j}$ and $\Delta^2 d_{i,j}$, as established in Proposition~\ref{lem:quad_dual}, and  based on these two key propositions, we prove our M-duality theorem.

Lastly, leveraging this M-duality theorem, we identify new accelerated methods that  exhibit a faster theoretical convergence rate than the Halpern iteration, which was previously known as the state-of-the-art method in this setting.

\subsection{Preliminaries}
\noindent\textbf{Notation.} Write $(X, \|\cdot\|)$ for a normed space. Write $\cB$ for a nonempty convex domain in $(X, \|\cdot\|)$.
An operator $\opT:\cB\rightarrow\cB$ is $\ell$-Lipschitz if $\|\opT x-\opT y\| \le \ell\|x-y\| $  for $\ell \in (0,1]$. For a given $\ell \in (0,1]$, define the class of Lipschitz operators as 
\[\cL_\ell =\{\opT:\cB\rightarrow\cB \,|\, \|\opT x-\opT y\| \le \ell\|x-y\|,  \forall x,y \in \cB, \cB\subset X, (X, \|\cdot\|) \text{ is a normed space\,}\}.\] 

 \vspace{0.2in}

\noindent\textbf{Fixed-point iteration.}
Given an operator $\opT: \cB \rightarrow \cB$, the classical Banach fixed-point theorem \cite{banach1922operations} states that if $\opT$ is contractive ($\ell<1$), $X$ is complete, and $\cB$ is closed, then a fixed point exists ($\opT x^\star- x^\star=0$). The following \emph{Picard iteration}
\[x^{n}=\opT x^{n-1}\qquad\text{ for } n=1,2,\dots,\]
converges to the unique fixed point of $\opT$ at a linear rate. If $\opT$ is nonexpansive ($\ell=1$) but not contractive, as in the case of a rotation operator, Picard iteration may not converge to a fixed point. In such cases, one may use \emph{Krasnosel'ski\u\i-Mann iteration} \cite{Mann1953mean, krasnosel1955two} (KM iteration)
\[ x^{n}=m^n x^{n-1}+ (1-m^{n}) \opT x^{n-1}
\qquad\text{ for } n=1,\dots,N\]
where $\{m^n\}_{n \in \mathbf{N}} \subset [0,1]$, or \emph{Halpern iteration} \cite{halpern1967fixed}
\[ x^{n}= (1-m^n) x^0+ m^{n} \opT x^{n-1}
\qquad\text{ for } n=1,\dots, N\]
where $x^0$ is an initial point and $\{m^n\}_{n \in \mathbf{N}}  \subset  [0,1]$ to improve convergence.

More generally, given an operator $\opT: \cB \rightarrow \cB$, $N \in \NN$, and a triangular array of averaging scalars $\{m^n_i\}_{0\le i \le n \le N} \subset [0,1]$
satisfying $ \sum_{i=0}^{n}m_i^n=1 \text{ for} \, 0\le n\le N,$ 
consider the \emph{Mann iteration}
\[
x^n=\sum_{i=0}^n m_i^n\,\opT x^{i-1},
\qquad n=0,\dots,N.
\]
Here, $x^0\in\cB$ and as a notational convention, we write $\opT x^{-1}=x^0$. Obviously, this Mann iteration includes Halpern iteration and KM iteration.

\vspace{0.2in}

\noindent\textbf{Optimal transport metrics framework and a tight residual bound.}
Following the analysis in \cite{baillon1992optimal, aygen1995optimal, bc2018, bravo2021universal, cc2023, bravo2026minimax, bravo2026krasnosel},
we introduce a nested family of optimal transport problems that will allow us to derive tight bounds for the residuals when an array of averaging scalars
$\{m^n_i\}_{0\le i \le n \le N}$ is given.

    Let $\mathcal{F}_{i,j}$ denote the set of transport plans from $\{m_k^i\}_{k=0,\dots,i}$ to $\{m_l^j\}_{l=0,\dots,j}$. That is, $\mathcal{F}_{i,j}$ consists of arrays $\{z_{k,l}\}_{0\le k\le i,\ 0\le l\le j}$ satisfying $z_{k,l}\ge0$ for all $k,l$ and
\begin{align*}
    \mbox{$\sum^j_{l=0} z_{k,l}$} = m^i_k &\mbox{ for all } k=0,\dots, i\\[1.5ex]
        \mbox{$\sum^i_{k=0} z_{k,l}$} = m^j_l & \mbox{ for all } l=0,\dots, j.
\end{align*}
    
Given $\ell \in (0,1]$ and $N\in \NN$, we define
\[R_N\triangleq\sum^N_{n=0} m^N_n \ell d_{n-1,N}\] where $\{d_{i,j}\}_{\{0\le i,j \le N\}}$ is defined recursively by the nested family of optimal transport problems
\[ d_{i,j} \triangleq \min_{z \in \mathcal{F}_{i,j}}\sum^i_{k=0}\sum^j_{l=0}  z_{k,l}\, \ell d_{k-1,l-1} \quad \forall\, 0\le i,j\le N\] 
starting with $d_{-1,i}=d_{i,-1}=1/\ell$ for all $0 \le i\le N$ and $d_{-1,-1}=0$. Also, define
\[\kappa=\max_{-1\le i,j\le N }{\|\opT x^i\!-\opT x^j\|}.\] 
Note that if $\opT$ has a fixed point, i.e., $x^*\!=\opT x^*\!$, one can show inductively that $\|x^n\!-x^*\|\leq\|x^0\!-x^*\|$ and $\kappa \le (1\!+\!\ell)\|x^0\!-x^*\|$. If $\opT$ has a bounded range, then $\kappa \le \|x^0\!-\opT x^0\|+\text{diam}(\opT(\cB))$. Also, if $\cB$ is bounded, then $\kappa \le \text{diam}(\cB)$.
This a priori estimate $\kappa$ will be used to establish tight upper bounds of the form $\|x^n\!-\opT x^n\|\leq\kappa\,R_n$ for the fixed-point residuals.

We now introduce several facts that show the relationship between the optimal transport metrics framework and the residual bound of a given Mann iteration.

\begin{fact}[Proposition~9, \cite{bravo2026minimax}]\label{fact:1}
 Let $\{x^n\}_{0\le n\le N}$ be a sequence generated by Mann iteration with $\opT\in\cL_\ell$ and coefficients $\{m^n_i\}_{0 \le i \le n \le N}$. Then, $\|x^i-x^j\|\leq \kappa\,d_{i,j}$ and $\|\opT x^N-x^N\|\le \kappa\,R_N$ for all $0 \le i,j\le N$.
\end{fact}

The following calculation gives a proof sketch of Fact~\ref{fact:1}.
\begin{align*}
\|x^{i}-x^{j}\|
&=\Bigl\|\sum^i_{k=0}m^i_k\,\opT x^{k-1}-\sum^j_{l=0}m^j_l\,\opT x^{l-1}\Bigr\|\\
&= \Bigl\| \sum_{k=0}^{i}\sum_{l=0}^{j} z_{k,l}\bigl(\opT x^{k-1}-\opT x^{l-1}\bigr) \Bigr\| \\
&\le  \sum_{k=0}^{i}\sum_{l=0}^{j} z_{k,l}\bigl\|\opT x^{k-1}-\opT x^{l-1}\bigr\|\\
&\le \begin{aligned}[t]
&\sum_{k=1}^{i}\sum_{l=1}^{j} z_{k,l}\, \ell \|x^{k-1}-x^{l-1}\|+\sum_{k=1}^{i} z_{k,0}\kappa+\sum_{l=1}^{j} z_{0,l}\kappa
\end{aligned}
\end{align*}
where the second equality follows from the transport plan $\{z_{k,l}\}_{0\le k\le i,\ 0\le l\le j}$ from $\{m^i_{k}\}_{k=0,\dots, i}$ to $\{m^j_l\}_{l=0,\dots, j}$,
the first inequality follows from the triangle inequality for the norm,
and the second inequality follows from the Lipschitz continuity of \(\opT\) and the definition of $\kappa$.
Then, by induction, \(\|x^{i-1}-x^{j-1}\| \le \kappa d_{i-1,\,j-1}\), we obtain
\begin{align*}
\|x^{i}-x^{j}\|
&\le \sum_{k=0}^{i}\sum_{l=0}^{j} z_{k,l}\,\ell \kappa d_{k-1,\,l-1}
\end{align*}
Then, the definition of $d_{i,j}$ yields the desired result. By a similar argument,
\begin{align*}
\|x^N-\opT x^N\|
= \left\|\sum^N_{n=0}m^N_n(\opT x^{n-1}-\opT x^N) \right\|
\le \kappa R_N.
\end{align*}

Moreover, the following fact shows that this bound is tight by constructing a worst-case operator.
\begin{fact}[Theorem~10, \cite{bravo2026minimax}]\label{fact:tightLip}
    For every $\kappa>0$, $\ell \in (0,1]$, and $\{m_i^n\}_{0\le i\le n\le N}$, there is an operator $\opT\in\cL_\ell$ on a normed space $(X, \|\cdot\|)$ and a corresponding Mann sequence $\{x^n\}_{0\le n\le N}$ such that $\|x^i-x^j\|=\kappa d_{i,j}$ and $\|\opT x^i-x^i\|=\kappa R_i$ for all $0\le i,j\le N$.
\end{fact}
In other words, optimal transport bounds \(d_{i,j}\) and
\(R_{N}\) are universal in the sense that they depend only on the coefficients $\{m^n_i\}_{0\le i \le n \le N}$ and the Lipschitz coefficient $\ell$, and not on the particular map \(\opT\) or the initial point \(x^{0}\). Hence, as mentioned in the Introduction, 
\[
R(\{m^n_i\}_{0 \le i \le n \le N}, \ell )=
\sup_{\substack{
\opT\in\cL_\ell,\ x^0\in\cB\\
x^n=\sum^n_{i=0}m^n_i\,\opT x^{i-1}\forall\,0\le n\le N
}}
\frac{\|x^N\!-\opT x^N\|}{\kappa}
\]
and $R(\{m^n_i\}_{0 \le i \le n \le N}, \ell)=R_N$. So, when $\{m^n_i\}_{0 \le i \le n \le N}$ and $\ell$ are given, we use $R_N$ and $R(\{m^n_i\}_{0 \le i \le n \le N}, \ell)$ interchangeably.

However, the preceding facts do not provide an explicit form of the optimal transport plan $z$ since they do not demonstrate how to calculate $z_{ij}$ when $\{m^n_i\}_{0\le i \le n \le N}$ and $\ell$ are given. 

\begin{fact}[Lemma~3, \cite{bravo2026minimax}]\label{metric}
    For $-1 \le i,j \le N$, $d_{i,j}$ is a metric.
\end{fact}
\begin{fact}[Lemma~4 \cite{bravo2026minimax}]\label{sim_opt}
    For all $0\leq i \leq j$, there is an optimal transport plan for $d_{i,j}$ such that $z_{k,k}=\min\{m_k^i,m_k^j\}$.
\end{fact}

Although Fact~\ref{sim_opt} from prior work \cite{bravo2026minimax} provides the explicit form of $z_{k,k}$, $z_{i,j}$ for $i\neq j$ still need to be specified.  In the next section, we derive this explicit form of $z_{i,j}$ for a given Mann iteration.

\section{Tightness of the residual bound for a Lipschitz operator }\label{sec:2}
Given  $\opT\in \cL_\ell$ with $\ell \in(0,1]$ and nonnegative averaging scalars 
$\{m^n_i\}_{0 \le i \le n \le N}$ satisfying $\sum^n_{i=0}m^n_i=1, m^n_n>0$,
consider the \emph{Mann iteration} 
 \[x^{n} = \mbox{$\sum^{n}_{i=0}  m^n_i\opT x^{i-1}$},    \qquad\forall\, n=0,\dots,N, \]
where $\opT x^{-1}=x^0 \in \cB$ and $N$ is the total iteration count. As mentioned in the Introduction, we consider the following assumptions on the coefficients.

\vspace{0.2in}

\noindent\textit{Assumption 1} (Monotonicity) For all $0 \le i < j \le N$, $m^j_i \le m^{j-1}_i$.

\vspace{0.1in}

\noindent\textit{Assumption 2} (Simple optimal transport) For all $0 \le i < j \le N$, $\sum^{j-1}_{l=i}m^j_l \le m^i_i$.

\vspace{0.2in}

Note that Assumption \ref{assum:sim} is automatically satisfied if $1/2 \le m^i_i$ for all $0 \le i \le N$.

Also, as explained in the preliminaries,
\begin{align*}
  &R(\{m^n_i\}_{0 \le i \le n \le N}, \ell ) = \sum^N_{n=0} m^N_n \ell d_{n-1,N},\\
  &d_{i,j} =\min_{z \in \mathcal{F}_{i,j}}\sum^i_{k=0}\sum^j_{l=0}  z_{k,l}\, \ell d_{k-1,l-1} \quad \forall\, 0\le i,j\le N,
\end{align*} 
where $d_{-1,i}=d_{i,-1}=1/\ell$ for all $0 \le i\le N$ and $d_{-1,-1}=0$. However, to actually compute the $R(\{m^n_i\}_{0 \le i \le n \le N}, \ell ) $ when $\{m^n_i\}_{0 \le i \le n \le N}$ and $\ell$ are given, we should solve optimal transport problems in the definition of $d_{i,j}$ and obtain the values of $z_{k,l}$.  

To establish this optimal transport plan, we define the \emph{consecutive quadrangle term} $\Delta^2 d_{i,j}$, which plays a crucial role in proving the tight bound and M-duality:
\[\Delta^2 d_{i,j} =\Delta d_{i,j} - \Delta d_{i,j-1} \qquad  \text{for all} \,\,0 \le i< j\le N,\]
where 
\[\Delta d_{i,j} = d_{i,j} - d_{i-1,j} \qquad  \text{for all} \,\,0 \le i \le j\le N.\]
Then, showing that $\Delta^2 d_{i,j} \ge 0$ for all $0 \le i< j\le N$ is equivalent to proving the \emph{convex quadrangle inequality} \cite{aygen1995optimal, bravo2021universal} as the following fact shows.
\begin{fact}[Lemma~5.1 \cite{bravo2021universal}]
    $d_{i,l}+d_{j,k} \le d_{i,k}+d_{j,l}$ for all $-1\le i\le j\le k \le l \le N$ if and only if $\Delta^2 d_{i,j} \ge 0$ for all $0 \le i< j\le N$.
\end{fact}

The inequality $d_{i,l}+d_{j,k} \le d_{i,k}+d_{j,l}$ for all $-1\le i\le j\le k \le l \le N$ is called the convex quadrangle inequality, which is studied as a necessary condition in the optimal transport metric framework for establishing the tightness of the bound
\cite{bravo2021universal, cc2023, bc2018, aygen1995optimal}. For the case $\ell=1$, this inequality was proved by \cite{bravo2021universal, aygen1995optimal}.

We now introduce our first key proposition which proves the quadrangle inequality for $\ell\in(0,1]$ and furthermore guarantees a simple optimal transport plan. Although this result can be proved by extending the technique used in the nonexpansive case $\ell=1$ to the contractive case, as noted in \cite{bravo2026krasnosel}, we present an alternative proof that explicitly computes the recursive formulas of the consecutive quadrangle term $\Delta^2 d_{i,j}$. 

\begin{proposition}\label{lem:sim}
    Under Assumptions \ref{assum:montone} and \ref{assum:sim},  for all $0 \le i<j\le N$, $\Delta^2 d_{i,j} \ge 0$ and
    \[d_{i,j} = \sum^{j-1}_{k=i+1}m^j_k \ell d_{i-1,k-1}+(m^i_i-\sum^{j-1}_{k=i}m^j_k)\ell d_{i-1,j-1} +\sum^{i-1}_{k=0}(m^i_k-m^j_k) \ell d_{k-1,j-1} .\]
\end{proposition}
\begin{proof}
We use induction on $j$. When $j=1$, Fact~\ref{sim_opt} gives
$d_{0,1}=(m_0^0 - m_0^1)\ell d_{-1,0}=m_1^1$.
We also have $\Delta d_{0,1}=m_1^1-\frac{1}{\ell}\le 0$ because
$m_1^1\le 1\le\frac{1}{\ell}$, and
$\Delta^2 d_{0,1}=m_1^1-\frac{1}{\ell}
-(0-\frac{1}{\ell})=m_1^1\ge 0$.

Now consider $j=N$ and the following transport plan:
\[
 z^{i,j}_{k,l}
=
\begin{cases}
\displaystyle
m^j_k
&\text{if } 0 \le l=k \le i \\[10pt]

\displaystyle
m^j_l
&\text{if } k=i,\, i+1 \le l \le j-1\\[10pt]

m^i_k-m^j_k
&\text{if }   \, 0 \le k \le i-1, l=j \\[10pt]

m^i_i-\sum^{j-1}_{l=i}m^j_l
&\text{if }  k=i, \, l=j \\[10pt]

0
& \text{otherwise.}
\end{cases}
\]
Note that $z^{i,j}_{k,l}$ is nonnegative by Assumptions~\ref{assum:montone} and~\ref{assum:sim}. Following the proof in \cite{cc2023}, consider the dual problem

\begin{align*}
  d_{i,j} &=  \max_{u, v} \sum^j_{l=0}u_l m^j_l-\sum^i_{k=0}v_k m^i_k\\
    & \quad \text{s.t.} u_l-v_k \le \ell d_{k-1,l-1} \qquad (0\le k\le i,\ 0\le l\le j)
 \end{align*}
and the following dual solution:
\begin{align*}
             &u_k=v_k=1-\ell d_{k-1,j-1}           \qquad k=0,\dots, i
             \\&u_l=1-\ell d_{i-1,j-1}+\ell d_{i-1,l-1}           \qquad  l=i+1,\dots, j
\end{align*}
For $k,l \le i$, the triangle inequality (Fact~\ref{metric}) gives
\[u_l-v_k = \ell d_{k-1,j-1}-\ell d_{l-1,j-1} \le \ell d_{k-1,l-1}.\]
For $k \le i <l \le j$, the induction hypothesis gives
\[u_l-v_k =\ell d_{i-1,l-1}- \ell d_{i-1,j-1}+\ell d_{k-1,j-1} \le \ell d_{k-1,l-1}.\]
Therefore, the proposed dual solution is feasible. By substituting the primal and dual solutions, we can verify that strong duality holds and that $z^{i,j}$ is an optimal transport plan.

Next, we prove that $\Delta d_{i,j} \le 0$ and $\Delta^2d_{i,j} \ge 0$.
By the preceding optimal transport formula, we have
\[
d_{i,j} =  \sum_{l=0}^{j} \sum_{k=0}^{i} z^{i,j}_{k,l} \ell d_{k-1,l-1}.
\]
This implies
\begin{align*}
    d_{i,j}-d_{i-1,j} = \sum_{l=0}^{j} \sum_{k=0}^{i} (z^{i,j}_{k,l}-z^{i-1,j}_{k,l}) \ell d_{k-1,l-1}
\end{align*}
where $z^{i-1,j}_{i,l}=0$ for all $l$.
For each $l$, by considering the change of basis from $\{d_{k,l}\}^{i-1}_{k=-1}$ to $\{d_{k,l}-d_{k-1,l}\}^{i-1}_{k=0}, d_{-1,l}$, we obtain
\[
\Delta d_{i,j} =  \sum_{l=0}^{j} \sum_{k=1}^{i} \Delta z^{i,j}_{k,l} \ell \Delta d_{k-1,l-1} + \sum_{l=0}^{j} \Delta z^{i,j}_{0,l}  \ell d_{-1,l-1}
\]

where
\[
\Delta z^{i,j}_{k,l}
=
\sum_{s=k}^{i} \left( z^{i,j}_{s,l} - z^{i-1,j}_{s,l} \right).
\]

Using the values of $z_{k,l}$, we can calculate \(\Delta z^{i,j}_{k,l}\) as follows:
\[
\Delta z^{i,j}_{k,l}
=
\begin{cases}
\displaystyle
m^i_i-\displaystyle\sum_{t=i}^{j-1} m^{j}_{t}
&\text{if } k=i,  l = j \\[10pt]

\displaystyle
\sum_{s=k}^{i-1}\left(m^{i}_s-m^{i-1}_s\right)+m^i_i
&\text{if } 1 \le k \le i-1,\ l=j\\[10pt]

m^j_{l}
&\text{if }  k=i, \, i \le l \le j-1 \\[10pt]

0
& \text{otherwise.}
\end{cases}
\]
Case (i). The case $k=i, l=j$ follows directly.

\noindent Case (ii). For $1 \le k \le i-1$ and $l=j$, we have
\[\sum_{s=k}^{i} z^{i,j}_{s,j} =\sum^{i-1}_{s=k}m^i_s - \sum^{i-1}_{s=k}m^j_s+m^{i}_i-\sum^{j-1}_{t=i}m^j_t=\sum^{i}_{s=k}m^i_s - \sum^{j-1}_{s=k}m^j_s.\] 
Thus, we obtain $\Delta z^{i,j}_{k,l}=\sum^{i}_{s=k}m^i_s-\sum^{i-1}_{s=k}m^{i-1}_s$.

\noindent Case (iii). The case $k=i$ and $i \le l \le j-1$ follows directly.

\noindent Case (iv). Otherwise, the corresponding terms coincide for $l\le i-1$, while for $i\le l\le j-1$, the differences $-m_l^j$ at $s=i-1$ and $m_l^j$ at $s=i$ cancel. The case $k=0,\ l=j$ follows from the equality of the column marginals. Hence, $\Delta z_{k,l}^{i,j}=0$.

Since $\Delta z^{i,j}_{0,l}=0$,  we have 
\[
\Delta d_{i,j} =  \sum_{l=0}^{j} \sum_{k=1}^{i} \Delta z^{i,j}_{k,l} \ell \Delta d_{k-1,l-1}.
\]

By Assumptions~\ref{assum:montone} and~\ref{assum:sim}, $\Delta z^{i,j}_{k,l} \ge 0$. Hence, induction yields $\Delta d_{i,j} \le 0$ for $0 < i<j$.

Then, we have 
\begin{align*}
    \Delta d_{i,j} - \Delta d_{i,j-1} =   \sum_{l=0}^{j} \sum_{k=1}^{i} (\Delta z^{i,j}_{k,l}-\Delta z^{i,j-1}_{k,l})\ell \Delta d_{k-1,l-1}
\end{align*}

where $\Delta z^{i,j-1}_{k,j}=0$ for all $k$. For each $k$, by considering the change of basis from $\{\Delta d_{k,l}\}^{j-1}_{l=-1}$ to $\{\Delta d_{k,l}-\Delta d_{k,l-1}\}^{j-1}_{l=0}, \Delta d_{k,-1}$, we obtain
\[
\Delta^2 d_{i,j} = \sum_{l=1}^{j} \sum_{k=1}^{i} \Delta^2 z^{i,j}_{k,l}\ell \Delta^2 d_{k-1,l-1} +\sum_{k=1}^{i} \Delta^2 z^{i,j}_{k,0}\ell \Delta d_{k-1,-1} 
 \]
where
\[
\Delta^2 z^{i,j}_{k,l} = \sum_{t=l}^{j} (\Delta z^{i,j}_{k,t} - \Delta z^{i,j-1}_{k,t}).
\]


Using the values of $z_{k,l}$ and $\Delta z_{k,l}$, we can calculate \(\Delta^2 z^{i,j}_{k,l}\) as follows:
\[
\Delta^2 z^{i,j}_{k,l}
=
\begin{cases}
\displaystyle
\sum_{t=i}^{l-1}\left(m^{j-1}_t-m^{j}_t\right)
&\text{if } k=i,\ i\le l \le j-1 \\[10pt]

\displaystyle
\sum_{s=k}^{i-1}\left(m^{i}_s-m^{i-1}_s\right)+m^i_i
&\text{if } 1 \le k \le i-1,\ l=j\\[10pt]

m^i_i-\displaystyle\sum_{t=i}^{j-1} m^{j}_{t}
&\text{if } k=i,\ l=j \\[10pt]

0
& \text{otherwise.}
\end{cases}
\]
Case (i). If $k=i, l=j$ or $1 \le k \le i-1, l=j$, then $\Delta^2 z^{i,j}_{k,l}= \Delta z^{i,j}_{k,j}$.

\noindent Case (ii). If $k=i$ and $i\le l \le j-1$, then $\Delta^2 z^{i,j}_{k,l}= \sum_{t=l}^{j} (z^{i,j}_{i,t}-z^{i,j-1}_{i,t})$ because $z^{i-1,j}_{i,l}=0$ for all $l$. We have \[\sum_{t=l}^{j} z^{i,j}_{i,t} =\sum^{j-1}_{t=l}m^j_t +m^{i}_i-\sum^{j-1}_{t=i}m^j_t=m^i_i -\sum^{l-1}_{s=i}m^j_s.\]
Thus, we obtain $\Delta^2 z^{i,j}_{k,l}=\sum^{l-1}_{s=i}(m^{j-1}_s-m^j_s)$.

\noindent Case (iii). (1) If $k=i$ and $0\le l \le i-1$, then $\Delta z^{i,j}_{k,l}=\Delta z^{i,j-1}_{k,l}=0$. (2) If $k \le i-1$ and $l \le j-1$, then $\sum_{t=l}^{j} \Delta z^{i,j}_{k,t}=\sum_{t=l}^{j-1} \Delta z^{i,j-1}_{k,t}=\sum_{s=k}^{i-1}\left(m^{i}_s-m^{i-1}_s\right)+m^i_i$. Therefore, $\Delta^2 z^{i,j}_{k,l}=0$.

Finally, since $\Delta^2 z^{i,j}_{k,0}=0$,  we have 
\[\Delta^2 d_{i,j} = \sum_{l=1}^{j} \sum_{k=1}^{i} \Delta^2 z^{i,j}_{k,l} \ell \Delta^2 d_{k-1,l-1}\]
where
\[ \Delta^2 z^{i,j}_{k,l} = \sum_{s=k}^{i}\sum_{t=l}^{j}
\left[z^{i,j}_{s,t}-z^{i-1,j}_{s,t}-z^{i,j-1}_{s,t}+z^{i-1,j-1}_{s,t}
\right].\]

Finally, by Assumptions~\ref{assum:montone} and~\ref{assum:sim}, $\Delta^2 z^{i,j}_{k,l}\ge 0$ for $1\le k\neq l$ and $\Delta^2 z^{i,j}_{k,k}=0$ for $1\le k$. It then follows by induction that $\Delta^2 d_{i,j}=\sum_{l=1}^{j} \sum_{k=1}^{i} \Delta^2 z^{i,j}_{k,l}\ell \Delta^2 d_{k-1,l-1}\ge 0$ for $0<i<j=N$.
The case $i=0$ follows directly from $d_{0,j}=1-m_0^j$, since $\Delta^2 d_{0,j}=m_0^{j-1}-m_0^j\ge0$ by Assumption~\ref{assum:montone}.

\end{proof}


Based on this optimal transport plan in Proposition~\ref{lem:sim} and the recursive formula of $\Delta^2 d_{i,j}$ obtained in the proof, we will prove the M-duality theorem in the next section.


\section{M-duality theorem}
\subsection{M-dual Mann iteration}
Given $\opT\in \cL_\ell$ with $\ell \in(0,1]$ and nonnegative averaging scalars 
$\{m^n_i\}_{0 \le i \le n \le N}$ satisfying $\sum^n_{i=0}m^n_i=1, m^n_n>0$, we define the \emph{M-dual Mann iteration} as
\begin{align*}
\!
x^{n+1}= \mbox{$m^{\tiny{N-n}}_{N-n}$}\opT x^n + \sum^{n-1}_{j=0} \!\left(\sum^{n-j}_{i=0} m^{N-j}_{N-n+i}- 
\sum^{n-j-1}_{i=0} \!\!m^{N-j-1}_{N-n+i} \right)\opT x^{j}  +\left(1-\!\sum^n_{i=0} m^{N}_{N-i}\right)x^0
\end{align*}
for all $0 \le n \le N-1$. Here, we replace $m^i_0$ by $1-\sum^i_{j=1}m^i_j$ for all $0\le i \le N$.

The following are examples of coefficient correspondences for $N=2,3$, where the lower triangular matrix represents the coefficients of the Mann iteration. Specifically, the entry in the $i$-th row and $j$-th column of the matrix represents the coefficient of $\opT x^{j-2}$ in the $(i-1)$-th iteration.

\vspace{0.1in}
\noindent For $N=2$:
\[
\begin{bmatrix} 1 &0 &0 \\ 1-m^1_1 &m^1_1 &0 \\ 1-m^2_1-m^2_2 &m^2_1 &m^2_2 \end{bmatrix}\quad \rightarrow \quad \begin{bmatrix} 1 &0  &0  \\  1-m^2_2 &m^2_2  &0   
\\ 1-m^2_2-m^2_1 &m^2_2+m^2_1-m^1_1  &m^1_1 
\end{bmatrix}\]
For $N=3$:
\[ \begin{bmatrix} 1 &0  &0&0  \\   1-m^1_1 &m^1_1  &0 &0  
\\ 1-m^2_1-m^2_2  &m^2_1 &m^2_2&0 
\\1-m^3_1-m^3_2-m^3_3 &m^3_1 &m^3_2&m^3_3  \end{bmatrix} \rightarrow \begin{bmatrix}   1 &0  &0&0  \\ 1-m^3_3 &m^3_3  &0 &0  
\\ 1-\sum^3_{i=2}m^3_i &\sum^3_{i=2}m^3_i-m^2_2  &m^2_2&0 
\\1-\sum^3_{i=1}m^3_i &\sum^3_{i=1}m^3_i-\sum^2_{i=1}m^2_i &\sum^2_{i=1}m^2_i-m^1_1 &m^1_1\end{bmatrix}\]




Generally, given $\{m_i^j\}_{0\le i \le j\le N}$, define
\begin{align*}
&(m_i^j)^M=\sum_{k=0}^{j-i}m_{N-j+1+k}^{N-i+1}-\sum_{k=0}^{j-i-1}m_{N-j+1+k}^{N-i}, \quad  1\le i\le j\le N,\\
&(m_0^j)^M=1-\sum_{r=1}^{j}(m_r^j)^M,\qquad 0\le j\le N.
\end{align*}
We call $(m_i^j)^M$ the dual coefficient corresponding to $m_i^j$. 
Then, $\bigl((m_i^j)^M\bigr)^M=m_i^j$, and the following lemma holds.

\begin{lemma}\label{m_dual_coef}
The coefficients of the Mann iteration $\{(m_i^j)^M\}_{0\le i \le j\le N}$ satisfy Assumptions~\ref{assum:montone} and~\ref{assum:sim} if and only if $\{m_i^j\}_{0\le i \le j\le N}$ satisfies Assumptions~\ref{assum:montone} and~\ref{assum:sim}.
\end{lemma}

\begin{proof}
Assume that the original coefficients satisfy the two assumptions. For $1\le i<j\le N$, $(m_i^j)^M-(m_i^{j-1})^M=m_{N-j+1}^{N-i+1}-m_{N-j+1}^{N-i}\le0.$
Also, telescoping the definition gives $(m_0^j)^M=\sum_{r=0}^{N-j}m_r^N$, and hence $(m_0^j)^M-(m_0^{j-1})^M=- m_{N-j+1}^N \le 0.$
Thus Assumption~\ref{assum:montone} holds. Moreover, for $1\le i<j\le N$,
\begin{align*}
(m_i^i)^M-\sum_{l=i}^{j-1}(m_l^j)^M
&=m_{N-i+1}^{N-i+1}-\left(\sum_{k=0}^{j-i}m_{N-j+1+k}^{N-i+1}-m_{N-j+1}^{N-j+1}\right)
\\&=m_{N-j+1}^{N-j+1}-\sum_{k=N-j+1}^{N-i}m_k^{N-i+1} \ge 0.    
\end{align*}
This proves Assumption~\ref{assum:sim}; the case $i=0$ follows from nonnegativity and the row-sum identity. The converse follows from $\bigl((m_i^j)^M\bigr)^M=m_i^j$.
\end{proof}

Due to Lemma~\ref{m_dual_coef}, we can apply the framework developed in Section~\ref{sec:2} to the Mann iteration with dual coefficients $\{(m_i^j)^M\}_{0\le i \le j\le N}$.

\subsection{A new formula for the residual bound}

To prove the M-duality theorem, we now introduce a new formula
for the residual bound $R_N$ which takes the form of a different recurrence relation
involving $d_{i,j}$.

First, by expanding $d_{i,j}$ using the recurrence relation in
Proposition~\ref{lem:sim}, we can express $d_{i,j}$ as a polynomial in
$\{m_k^r\}_{0\le k\le r\le j}$ and $\ell$. We define $f_{i,j}$ as the first-order
term of $d_{i,j}$. Specifically, by Proposition~\ref{lem:sim}, $(m_0^i-m_0^j)\ell d_{-1,j-1}=m_0^i-m_0^j=f_{i,j}$.

Next, consider $R_N=m_0^N+\sum_{n=1}^N m_n^N\ell d_{n-1,N}$. Expand $d_{i,j}$ terms for $i,j>0$ using Proposition~\ref{lem:sim}, and collect the $f_{i,j}$ terms. 
Repeating this procedure iteratively, we can finally express $R_N$ as $m^N_0$ plus a linear combination of $f_{i,j}$. Let $c_{i,j}$ be the coefficient of $f_{i,j}$, which should be a polynomial in $m_k^i$ and $\ell$.  Then, we have the following lemma.



 \begin{lemma}\label{lem:diff-rec}
    \[R_N = 1-\sum^N_{j=1}m^N_j+\sum_{0\le i<j\le N} c_{i,j}f_{i,j}\]
where the $c_{i,j}$ are generated by the following recurrence relation:
\[c_{i,j}=\ell\sum_{l=j+1}^{N} \sum_{k=i+1}^{l-1} z^{k,l}_{i+1,j+1}c_{k,l},\qquad 0\le i<j<N,\]and \[c_{n,N}=\ell m^N_{n+1},\qquad 0\le n<N.\]
\end{lemma}
\begin{proof}
    By definition of $R_N$, $c_{n,N}=\ell m^N_{n+1}$ for $0\le n<N$. For $i<k, j<l$, observe that $f_{i,j}$ and $f_{k,l}$ appear one time in the expansions of $d_{i,j}$ and $d_{k,l}$, respectively, and $d_{i,j}$ appears $\ell z^{k,l}_{i+1,j+1}$ times in the expansion of $d_{k,l}$ by Proposition~\ref{lem:sim}. This implies that, with the definition of $c_{i,j}$, $f_{i,j}$ appears $\ell z^{k,l}_{i+1,j+1}c_{k,l}$ times in the one-step expansion of $d_{i,j}$ by Proposition~\ref{lem:sim}. Summing these over all $k>i$ and $l>j$ gives the stated recurrence. Lastly, since $l>j$, backward induction on $j=N,N-1,\dots,1$ shows that every $d$-term is eliminated after finitely many steps and that its coefficient is transferred to the corresponding $f$-term. 
\end{proof}


We now introduce our second key proposition which reveals a surprising relationship between $\Delta^2 d_{i,j}$ in Mann iteration and $(c_{i,j})^M$ in M-dual Mann iteration.
\begin{proposition}\label{lem:quad_dual}
Under Assumptions \ref{assum:montone} and \ref{assum:sim}, for every $0\le i<j\le N$, \[\ell\Delta^2 d_{i,j}=(c_{N-j,N-i})^M.\]
\end{proposition}
\begin{proof}
For $1\le i<j\le N$ (the case $i=0$ is treated below), we suppress the zero-cost diagonal entries of the plan in Proposition~\ref{lem:sim} and write
\[
 z^{i,j}_{k,l}
=
\begin{cases}
\displaystyle
m^j_l
&\text{if } k=i,\, i+1 \le l \le j-1\\[10pt]

m^i_k-m^j_k
&\text{if }   \, 0 \le k \le i-1, l=j \\[10pt]

m^i_i-\sum^{j-1}_{r=i}m^j_r
&\text{if }  k=i, \, l=j \\[10pt]

0
& \text{otherwise.}
\end{cases}
\]
For $1\le k\le i$ and $k<l\le j$, substitution gives the
corresponding M-dual entries:
\[
 (z^{i,j}_{k,l})^M
=
\begin{cases}
\displaystyle
\sum^{N-l+1}_{t=N-j+1}m^{N-l+1}_{t} -\sum^{N-l}_{t=N-j+1}m^{N-l}_{t}
&\text{if } k=i,\, i+1 \le l \le j-1\\[10pt]

\sum^{N-i}_{t=N-j+1} m^{N-k}_{t}-\sum^{N-i}_{t=N-j+1} m^{N-k+1}_{t} 
&\text{if }   \, 1 \le k \le i-1, l=j \\[10pt]

m^{N-j+1}_{N-j+1}  - \sum^{N-i}_{t=N-j+1} m^{N-i+1}_{t}
&\text{if }  k=i, \, l=j \\[10pt]

0
& \text{otherwise.}
\end{cases}
\]

In the proof of Proposition~\ref{lem:sim}, we showed that
    \[
\Delta^2 z^{i,j}_{k,l}
=
\begin{cases}
\displaystyle
\sum_{t=i}^{l-1}\left(m^{j-1}_t-m^{j}_t\right)
&\text{if } k=i,\ i+1\le l \le j-1 \\[10pt]

\displaystyle
\sum_{s=k}^{i-1}\left(m^{i}_s-m^{i-1}_s\right)+m^i_i
&\text{if } 1 \le k \le i-1,\ l=j\\[10pt]

m^i_i-\displaystyle\sum_{t=i}^{j-1} m^{j}_{t}
&\text{if } k=i,\ l=j \\[10pt]

0
& \text{otherwise.}
\end{cases}
\]
It follows that
\[
\Delta^2 z^{N-l+1,N-k+1}_{N-j+1,N-i+1}
=
\begin{cases}
\displaystyle
\sum_{t=N-j+1}^{N-i}\left(m^{N-k}_t-m^{N-k+1}_t\right)
&\text{if } l=j,\ 1\le k \le i-1 \\[10pt]

\displaystyle
\sum_{s=N-j+1}^{N-l}\left(m^{N-l+1}_s-m^{N-l}_s\right)+m^{N-l+1}_{N-l+1}
&\text{if } i+1 \le l \le j-1,\ k=i\\[10pt]

m^{N-j+1}_{N-j+1}-\displaystyle\sum_{t=N-j+1}^{N-i} m^{N-i+1}_{t}
&\text{if } k=i,\ l=j \\[10pt]

0
& \text{otherwise.}
\end{cases}
\]
Comparing this formula with the expression for $(z_{k,l}^{i,j})^M$ above gives
\[\Delta^2 z_{N-j+1,N-i+1}^{N-l+1,N-k+1} = (z_{k,l}^{i,j})^M \]
for all $1\le k\le i<j\le N$ and $k<l\le j$.
The previously deferred case $i=0$ follows immediately from the
terminal condition in Lemma~\ref{lem:diff-rec} and the definition of the M-dual, since
\[(c_{N-j,N})^M = \ell(m_{N-j+1}^N)^M = \ell(m_0^{j-1}-m_0^j) = \ell\Delta^2d_{0,j}.\]
Thus, \(\ell\Delta^2 d_{i,j}\) and \((c_{N-j,N-i})^M\) satisfy the same
recurrence and agree for \(i=0\). Hence, by induction,
\[ \ell\Delta^2d_{i,j}=(c_{N-j,N-i})^M \qquad \text{for}\, 0\le i<j\le N \]
which proves the claim.

\end{proof}

\subsection{M-duality theorem}
With Propositions~\ref{lem:diff-rec} and~\ref{lem:quad_dual},  we now prove the M-duality theorem. For completeness, we restate the theorem. 

\vspace{0.2in}
\noindent\textbf{Theorem 1} (M-duality theorem) Under Assumptions \ref{assum:montone} and \ref{assum:sim}, 
\vspace{-0.1in}
    \[R(\{m^n_i\}_{0 \le i \le n \le N}, \ell)
    =
    R(\{m^n_i\}^M_{0 \le i \le n \le N}, \ell).\]
\begin{proof}
First, by definition of $d_{i,j}$ and  Proposition~\ref{lem:quad_dual}, we have
\[\ell d_{i,N}=\ell\sum_{l=i+1}^{N}\sum_{k=0}^{i}\Delta^2d_{k,l}=\sum_{l=i+1}^{N}\sum_{k=0}^{i}(c_{N-l,N-k})^M\]
where the identities $d_{-1,l}=1/\ell$ and $d_{i,i}=0$ are used in the first equality. Therefore,
\begin{align*}
R_N &=m_0^N+\sum_{i=0}^{N-1}m_{i+1}^N\ell d_{i,N}\\
&=m_0^N+\sum_{0\le k<l\le N}(c_{N-l,N-k})^M\sum_{r=k+1}^{l}m_r^N\\
&=(m_0^N)^M+\sum_{0\le k<l\le N}(c_{N-l,N-k})^M(f_{N-l,N-k})^M\\
&=(R_N)^M,
\end{align*}
where the third equality is from $(m_0^N)^M=m_0^N$ and $(f_{N-l,N-k})^M=\sum_{r=k+1}^{l}m_r^N$ and the last equality comes from Lemma~\ref{lem:diff-rec}.
\end{proof}

As discussed in the Introduction, our Theorem~\ref{thm:m-dual} is the first duality theorem for fixed-point iterations in general normed spaces. The work \cite{yoon2024optimal} considers the Hilbert-space setting, while \cite{kim2023mirror} studies mirror descent in non-Euclidean space and focuses on the relationship between function values and gradient norms. 

 
\paragraph{Relationship with H-duality}
  \cite{yoon2024optimal} expresses fixed-point iteration as
$x^{k+1} = x^k-\sum^k_{j=0} h_{k+1,j+1}(x^j-\opT x^j)$,
which is called the $h$-stepsizes form, and H-duality indicates the antitranspose of the $H$-matrix, where $H \in \real^{N\times N}$ and $H_{i,k}=h_{k+1,i}$ if $0 \le i \le k \le N-1$ and $H_{k,i}=0$ otherwise. Through simple manipulation, it can be verified that the M-duality relationship is equivalent to the H-duality relationship. In other words, these are the same transformation between fixed-point iterations, but with different representations. So following the prior work, our lemmas and theorems could therefore be written with respect to the $h$-stepsizes form. However, we choose the Mann-iteration representation because we found that it is necessary for operations in normed spaces. In particular, considering the optimal transport metrics framework and the proofs of Propositions~\ref{lem:sim} and~\ref{lem:quad_dual}, we believe that the Mann-iteration expression is a natural representation that captures the mechanism underlying the symmetry of residual bounds between Mann iterations which are M-dual to each other.

\subsection{Application of the M-duality theorem to Mann iterations}
In this subsection, we apply Theorem~\ref{thm:m-dual} to KM and Halpern iterations. Recall that KM iteration is
\[ x^{n}=m^n x^{n-1}+ (1-m^{n}) \opT x^{n-1}
\qquad\text{ for } n=1,2,\dots,N\]
Then, its dual iteration is
\[ x^{n+1}=m^{N-n} x^{n}+ (1-m^{N-n}) \opT x^{n}
\qquad\text{ for } n=0,1,\dots,N-1\]
Interestingly, the dual of KM iteration is KM iteration with the coefficients in reverse order, and by Theorem~\ref{thm:m-dual}, they have the same $R_N$.
\begin{corollary}
       Under Assumptions~\ref{assum:montone} and~\ref{assum:sim}, the KM iterations with averaging factors $\{m^i\}^N_{i=1}$ and $\{m^{N+1-i}\}^N_{i=1}$ have the same residual bound $R_N$.
\end{corollary}

Next, we apply the M-duality theorem to Halpern iteration. Recall that Halpern iteration is
\[ x^{n}= (1-m^{n})x^{0}+ m^n \opT x^{n-1}
\qquad\text{ for } n=1,2,\dots,N.\]
We call its dual the \emph{Dual-Halpern iteration}, which has also been studied in Hilbert spaces \cite{yoon2024optimal, yoon2026toward}.

\vspace{0.3in}

\textbf{Dual-Halpern iteration}
\begin{align*}
    x^{n+1} &= m^{N-n} \opT x^{n} +\sum^{n-1}_{i=0} (m^{N-i}-m^{N-i-1})  \opT x^{i}+(1-m^N)x^0
    \\&= x^{n} + m^{N-n}( \opT x^{n}- \opT x^{n-1}) \qquad\text{ for } n=0,1,\dots,N-1
\end{align*}
Although Halpern and Dual-Halpern iterations have seemingly different forms of iteration,  by the M-duality theorem, they have the same $R_N$.
\begin{corollary}
   Under Assumptions~\ref{assum:montone} and~\ref{assum:sim}, the Halpern iteration with anchoring factors $\{m^i\}^N_{i=1}$ and the corresponding Dual-Halpern iteration have the same  residual bound $R_N$.
\end{corollary}
Recent work \cite{bravo2026minimax} showed that the coefficients $m^0=0, m^n= \min \{1, \frac{1}{2\ell} (1+(\ell m^{n-1})^2)\}$ yield minimax optimality for Halpern iteration. 
The M-duality theorem implies that this choice also yields minimax optimality for Dual-Halpern iteration under Assumptions~\ref{assum:montone} and~\ref{assum:sim}, because the corollary implies that the corresponding Dual-Halpern iteration shares $R_N$ with Halpern iteration.

\section{A new method for efficiently reducing the residual}\label{sec4}

In Hilbert spaces, Halpern iteration is known to be optimal by matching the
corresponding lower bound \cite{pr2022}. In general normed spaces, surprisingly, however,
the M-dual representation suggests that a small modification of the boundary
coefficients can strictly improve the residual bound. Now, we introduce our novel algorithm \emph{Perturbed Dual-Halpern iteration}.

\paragraph{Perturbed Dual-Halpern iteration}
\begin{align*}
x^1
&=(1-m^N+\delta_1)x^0+(m^N-\delta_1)\opT x^0,\\
x^n
&=(1-m^N)x^0
+\sum_{i=0}^{n-2}(m^{N-i}-m^{N-i-1})\opT x^i
+m^{N+1-n}\opT x^{n-1} \qquad 2\le n\le N-1,\\
x^N
&=(1-m^N-\delta_2)x^0
+\sum_{i=0}^{N-2}(m^{N-i}-m^{N-i-1})\opT x^i
+(m^1+\delta_2)\opT x^{N-1}.
\end{align*}


The Perturbed Dual-Halpern iteration replaces $m^N$ by $m^N-\delta_1$ in the first iteration and $m^1$ by $m^1+\delta_2$ in the last iteration, with corresponding adjustments to the coefficients of $x^0$ and leaving the intermediate coefficient rows unchanged. As the following theorem shows, combining this method with Picard iteration improves the residual bound achieved by the minimax-optimal Halpern and Dual-Halpern iterations.




\begin{theorem}
\label{thm:perturbed-dual-halpern}

Let $N \ge 2$ and $\frac{5}{8}<\ell\le 1$.
Let $m^0=0$, 
\[m^n=\min\left\{1,\frac{1+(\ell m^{n-1})^2}{2\ell}\right\}, \qquad  n_*=\inf\{n\ge 1:m^{n+1}=1\}, \qquad N_*=\min\{n_*,N\}.\]
Let
 \[\delta_{2}= \frac{1-m^{N_*}}{8\ell }>0, \qquad \delta_1=\frac{3}{16\ell N_*(N_*+1)}>0.\]


Then, under Assumptions~\ref{assum:montone} and~\ref{assum:sim}, running the Perturbed Dual-Halpern iteration for $N_*$ iterations, followed by $N-N_*$ Picard iterations, yields
\[ R_N \le R_N^{\mathrm H} -\ell^{N-N_*+1}\left(\frac{\delta_2^2}{3}+\delta_1^2\right),\]
where $R_N^{\mathrm H}$ is the residual bound for Halpern iteration with anchoring factors $\{m^i\}^N_{i=1}$.
\end{theorem}

\begin{proof}
First we consider $N$ such that $m^N<1$.

First, consider the case $\delta_1=0$. Assumptions~\ref{assum:montone} and~\ref{assum:sim} remain valid for $0\le\delta_2\le1-m^N$. Also, Proposition~\ref{lem:sim} gives
\[R_N(\delta_2,0)=R_N^{\mathrm H}-g_N\delta_2+\ell\delta_2^2,\]
where $g_N:=1-\ell m^N-\ell^N\prod_{r=1}^{N}m^r$ so the $N$th-iteration perturbation gives the decrease $g_N\delta_2-\ell\delta_2^2$.
Then, we can prove the following lemma.
\begin{lemma}
Let \(5/8<\ell\le 1\) and \(N\ge 2\), and assume that \(m^N<1\). Then
\[
g_N-\frac{1-m^N}{8}
\ge 0.
\]
\end{lemma}
\begin{proof}
Since $\frac{1+x_{n-1}^2}{2}$ with $x_0=0$ increases with $n$, by substituting $x_n:=\ell m^n < \ell $, we have
\[
x_n
=\frac{1+x_{n-1}^2}{2}.
\]

Let
\[
X_n:=\prod_{r=1}^n x_r,
\qquad
P_n:=\frac{X_n}{1-x_n}.
\]
Since
\[
1-x_{n+1}
=
1-\frac{1+x_n^2}{2}
=
\frac{(1-x_n)(1+x_n)}{2},
\]
we have, for \(2\le n<N\),
\[
\begin{aligned}
\frac{P_{n+1}}{P_n}
=
\frac{x_{n+1}(1-x_n)}{1-x_{n+1}} =
\frac{1+x_n^2}{1+x_n}
<1,
\end{aligned}
\]
where the last inequality follows from \(0<x_n<\ell\).

Moreover, since
\[
P_2
=
\frac{x_1x_2}{1-x_2}
=
\frac{\frac12\cdot\frac58}{1-\frac58}
=
\frac56,
\]
we have $P_N \le \frac{5}{6}$. Also,
\[
g_N
=
1-x_N-X_N
=
(1-x_N)(1-P_N)
\]
 implies
\[
\begin{aligned}
g_N-\frac{1-m^N}{6}
&=
(1-x_N)(1-P_N)-\frac{1-m^N}{6} \\
&=
(1-x_N)\left(\frac56-P_N\right)
 +\frac{1-x_N-(1-m^N)}{6} \\
&=
(1-x_N)\left(\frac56-P_N\right)
 +\frac{(1-\ell)m^N}{6}.
\end{aligned}
\]
Both terms on the right-hand side are nonnegative. Hence
\[
g_N\ge \frac{1-m^N}{6}.
\]
This proves the claim.
\end{proof}
Let $c=\ell\delta_2.$ Since $\ell\delta_1=3/[16N(N+1)]<x_1=\ell m^1$, we have $\delta_1<m^1\le m^N$. So $m_0^1=1-m^N+\delta_1\le1$, and a direct verification using the monotonicity of $\{m^k\}$ shows that the first-row perturbation also preserves Assumptions~\ref{assum:montone} and~\ref{assum:sim}.

Now, we consider the case  $\delta_2=0$. Since $\delta_1$ appears only in the first iteration, Proposition~\ref{lem:sim} gives
\[d_{i,j}(\delta_1)=d_{i,j}(0)+u_{i,j}\delta_1+\ell v_{i,j}\delta_1^2\]
where
\[u_{0,1}=-1,\quad u_{0,j}=0\ (j\ge2),\quad u_{1,2}=1-x_{N-1}-x_N,\quad u_{1,j}=1+x_{N-2}-x_{N-1}-x_N\ (j\ge3).\]
For $2\le i<j\le N$, comparison of the linear coefficients gives
\[u_{i,j}=\sum_{q=i+1}^{j-1}(x_{N-q+1}-x_{N-q})u_{i-1,q-1}+x_{N-j+1}u_{i-1,j-1}.\]
A direct induction yields
\[
u_{i,j}=
\begin{cases}
\displaystyle \left(\prod_{r=N-i}^{N-2}x_r\right)(1-x_{N-1}-x_N),&j=i+1,\\[6pt]
\displaystyle \left(\prod_{r=N-i}^{N-2}x_r\right)(1+x_{N-i-1}-x_{N-1}-x_N),&j\ge i+2.
\end{cases}
\]
Let $a_N$ denote the negative linear coefficient of $R_N(0,\delta_1)$. Substitution of the preceding formula into the final residual gives
\[a_N=\sum_{t=1}^{N-1}(x_t-x_{t-1})\left(\prod_{r=t}^{N-2}x_r\right)(x_{N-1}+x_N-1-x_{t-1}).\]
For $k\ge2$, define $a_k$ by the same formula with $N$ replaced by $k$. A direct telescoping calculation using $x_{k+1}=(1+x_k^2)/2$ gives
\[a_2=\frac{1}{16},\qquad a_{k+1}-x_{k-1}a_k=\frac{(1-x_{k-1}^2)^3}{16}\ge0.\]
Consequently,
\[a_N\ge a_2\prod_{r=1}^{N-2}x_r=\frac{X_{N-2}}{16}.\]

We now restore the $N$th-row perturbation. In the recurrence of Proposition~\ref{lem:sim}, this perturbation adds $\delta_2\ell d_{-1,N-1}=\delta_2$ to every $d_{i,N}$. Hence the coefficients $u_{i,N}$ and $v_{i,N}$ are unchanged. The only additional $\delta_1$-dependent contribution to the final residual comes from the change of the last weight $x_1\mapsto x_1+c$. Since
\[u_{N-1,N}=X_{N-2}(1-x_{N-1}-x_N),\]
we obtain
\[R_N(\delta_2,\delta_1)=R_N(\delta_2,0)-\alpha_N\delta_1+\ell\beta_N\delta_1^2,\]
where
\[\alpha_N=a_N+(x_{N-1}+x_N-1)X_{N-2}c.\]
Since $N\ge2$ and $\{x_k\}$ is increasing, we have
$x_{N-1}+x_N\ge x_1+x_2=9/8>1$. Consequently,
\[\alpha_N\ge\frac{X_{N-2}}{16}.\]
It remains to bound the quadratic coefficient. We have $v_{0,j}=0$, $v_{1,2}=1$, and $v_{1,j}=0$ for $j\ge3$. For $i\ge2$, comparison of the quadratic coefficients gives
\[v_{i,j}=\sum_{q=i+1}^{j-1}(x_{N-q+1}-x_{N-q})v_{i-1,q-1}+x_{N-j+1}v_{i-1,j-1}.\]
The coefficients are nonnegative and their sum is $x_{N-i}\le1$. Hence induction gives $0\le v_{i,j}\le1$, and therefore
\[0\le\beta_N\le\sum_{i=1}^{N-2}(x_{N-i}-x_{N-i-1})+(x_1+c)=x_{N-1}+c.\]
Since $c\le\ell-x_N$ and $x_{N-1}\le x_N$, we obtain
\[0\le\beta_N\le\ell\le1.\]

Finally, induction gives
\[x_k\ge\frac{k+1}{k+3}.\]
Thus,
\[X_{N-2}\ge\prod_{k=1}^{N-2}\frac{k+1}{k+3}=\frac{6}{N(N+1)}.\]
Combining the preceding bounds gives
\[R_N(\delta_2,\delta_1)\le R_N(\delta_2,0)-\frac{3}{8N(N+1)}\delta_1+\ell\delta_1^2.\]
With $\delta_1=3/[16\ell N(N+1)]$, we obtain
\[R_N(\delta_2,0)-R_N(\delta_2,\delta_1)\ge\frac{9}{256\ell N^2(N+1)^2}.\]
Combining this decrease with the $N$th-row decrease proves the claim.

If $N>N^*$ we run Picard, and have $\ell^{N-N^*}R_{n^*}$. Since the minimax-optimal Halpern iteration also turns into Picard after $N^*$ iterations and $R^H_n=\ell^{N-N^*}R^H_{N^*}$ \cite[Section~2.2]{bravo2026minimax}, this concludes the proof.
\end{proof}

For the non-expansive case $\ell=1$, we have $N_*=N$, and thus the Perturbed Dual-Halpern iteration does not reduce to the Picard iteration.

We note that \cite{cc2023} conducted numerical optimization for designing Mann iterations utilizing the optimal transport metrics framework and found numerical evidence that the Halpern iteration is not an optimal method in general normed spaces. However, to the best of our knowledge, no analytic method has been proposed that achieves a residual bound equal to or strictly better than that of the Halpern iteration. Our Perturbed Dual-Halpern iteration is the first analytic algorithm to exhibit a strictly smaller worst-case residual bound than the Halpern iteration as stated in Theorem~\ref{thm:perturbed-dual-halpern}.

Lastly, by the M-duality theorem, we directly obtain the M-dual of Perturbed Dual-Halpern, which has the same residual bound.

\paragraph{Dual Perturbed Dual-Halpern}

\begin{align*}
x^1&=(1-m^1-\delta_2)x^0 + (m^1+\delta_2) \opT x^0\\
     x^{n} &=(1-m^n-\delta_2)x^0 +\delta_2 \opT x^0+m^{n} \opT x^{n-1} \qquad 2\le n\le N-1,
    \\x^N &=  (1-m^N-\delta_2)x^0 +\delta_2 \opT x^0+\delta_1 \opT x^{N-2}+(m^{N}-\delta_1) \opT x^{N-1}. 
\end{align*}

\begin{corollary}
 Under Assumptions~\ref{assum:montone} and~\ref{assum:sim}, the Perturbed Dual-Halpern iteration and the corresponding Dual Perturbed Dual-Halpern iteration have the same  residual bound $R_N$.
\end{corollary}


\section{Conclusion}\label{sec13}
In this work, we establish an M-duality theorem for Lipschitz operators in general normed spaces within the optimal transport metric framework. One direction for future research is to relax the coefficient conditions in Assumption~\ref{assum:montone} or Assumption~\ref{assum:sim}. Another is to extend the theory to operators with Lipschitz constants $\ell\in(1,\infty)$ and investigate M-duality in the context of the minimal displacement problem~\cite{g1973, d2025, bravo2026minimax}.

Just as H-duality has motivated new algorithms and convergence analyses in Hilbert spaces, we hope that our M-duality theorem will open new directions in general normed spaces. In particular, as our construction of the Perturbed Dual-Halpern iteration has shown, we believe that M-duality could guide the design of algorithms with sharper residual bounds and ultimately help identify optimal fixed-point methods in general normed spaces.


\newpage









\bibliography{sn-bibliography}

\end{document}